\documentclass[a4paper]{amsart}

\usepackage{amsmath,amssymb, amsbsy}
  \usepackage{paralist}
  \usepackage{graphics} 
  \usepackage{epsfig} 
\usepackage{amsthm}
\usepackage{tikz}
\usepackage[all]{xy}
 \usepackage[colorlinks=true]{hyperref}
\hypersetup{urlcolor=blue, citecolor=red}

\newtheorem{teor}{Theorem}
\newtheorem{lemma}[teor]{Lemma}

{\theoremstyle{definition}

\newtheorem{Levi-Civita}[teor]{On the Levi-Civita transform}
\newtheorem{rem}[teor]{Remark}}

\newcommand{\R}{\mathbb{R}}

\def\beq{\begin{equation}}
\def\eeq{\end{equation}}

\def\pa{\partial}
\def\t{\theta}

\def\d{\delta}

\def\wt{\widetilde}
\def\wh{\widehat}
\def\f{\varphi}
\def\l{\lambda}

\def\n{\nabla}

\def\eps{\varepsilon}

\title{Addendum to: Symbolic dynamics for the $N$-centre problem at negative energies}

\author[Nicola Soave and Susanna Terracini]{}

\subjclass{Primary: 70F10, 37N05; Secondary: 70F15, 37J30.}
\keywords{$N$-centre problem, chaotic motions, symbolic dynamics, Levi-Civita regularization.}

 \email{n.soave@campus.unimib.it }
 \email{susanna.terracini@unimib.it}

\thanks{N. Soave and S. Terracini were
partially supported by the PRIN2009 grant ``Critical Point Theory and
Perturbative Methods for Nonlinear Differential
Equations''.}

\begin{document}
\maketitle
\centerline{\scshape Nicola Soave}
\medskip
{\footnotesize
 \centerline{Universit\`a di Milano Bicocca - Dipartimento di Ma\-t\-ema\-ti\-ca e Applicazioni}
   \centerline{Via Cozzi 53}
   \centerline{20125 Milano, Italy}
} 

\medskip

\centerline{\scshape Susanna Terracini}
\medskip
{\footnotesize
 \centerline{Universit\`a di Milano Bicocca - Dipartimento di Ma\-t\-ema\-ti\-ca e Applicazioni}
   \centerline{Via Cozzi 53}
   \centerline{20125 Milano, Italy}
} 

\begin{abstract}
This paper aims at completing and clarifying a delicate step in the proof of Theorem 5.3 of our paper \cite{ST}, where it was used the differentiability of a function $F$, which a priori can appear not necessarily differentiable.  
\end{abstract}

\section*{Introduction}
In Step 2) of the proof Theorem 5.3 of our paper \cite{ST}, we stated that the function $F$ has partial derivatives. We refer here to the notation of our  paper.
Actually this is not immediately granted, due to the lack of uniqueness of inner minimizer of the Maupertuis' functional $M$; however the quoted Theorem still holds true, and a posteriori also the differentiability of the function $F$. 
We can prove it with the introduction of a family of auxiliary smooth functions which are strictly related to $F$.

\section*{Addendum}

We refer to Step 2) in the proof of Theorem 5.3. of our paper \cite{ST}, which is the main reference for this paper. Let $k \in \{0,\ldots,2n-1\}$. To fix the ideas, let $k=2j+1$ for some $j \in \{0,\ldots,n-1\}$. We introduce a  neighbourhood $U_{2j+1}$ of the point $\bar{p}_{2j+1}$ which is strongly convex with respect to the Jacobi metric. Let us choose $t_* \in (0,T_{2j+1})$ such that 
\[
\wt{p}_{2j+1}:=y_{2j+1}(t_*) \in  U_{2j+1},\quad |\wt{p}_{2j+1}|<R, \quad y\left([0,t_*]\right) \subset \left(B_R(0) \setminus B_{R/2}(0)\right);
\]
in this way, in $[0,t_*]$, the function $y_{2j+1}$ does not interact with the singularities of the potential. There exists a unique minimal geodesic $\wh{y}(\cdot;\bar{p}_{2j+1},\wt{p}_{2j+1};\eps)$ for the Jacobi metric, parametrized with respect to the arc length, connecting $p_{2j+1}$ and $\wt{p}_{2j+1}$ and lying in $U_{2j+1}$, which depends smoothly on its ends. We know that $y_{2j+1}$ is a minimizer of the length $L$ connecting $p_{2j+1}$ and $p_{2j+2}$, therefore (Proposition 4.8) this geodesic has to be a reparametrization of $y_{2j+1}$. Note that if $p_{2j+1} \in \bar{U}_{2j+1}$, then there exists a unique minimal geodesics $\wh{y}(\cdot;p_{2j+1},\wt{p}_{2j+1};\eps)$ for the Jacobi metric, parametrized with respect to the arc length, which connects $p_{2j+1}$ and $\wt{p}_{2j+1}$. We will consider the reparametrization $\wt{y}(\cdot\,;p_{2j+1},\wt{p}_{2j+1};\eps)$ of $\wh{y}(\cdot\,;p_{2j+1},\wt{p}_{2j+1};\eps)$ such that
\[
\begin{cases}
\ddot{\wt{y}}(t)=\n V_\eps(\wt{y}(t)) \\
\frac{1}{2}|\wt{y}(t)|^2-V_\eps(\wt{y}(t))=-1,
\end{cases}
\]
denoting by $[0,T(p_{2j+1},\wt{p}_{2j+1})]$ its domain. Due to the minimality of 
$\wh{y}(\cdot\,;p_{2j+1},\wt{p}_{2j+1};\eps)$ for $L$, such a reparametrization exists, see Theorem 4.5. In this way
\beq\label{oss 1}
\wt{y}(\cdot\,;\bar{p}_{2j+1},\wt{p}_{2j+1};\eps)\equiv y_{P_{k_{j+1}}}(\cdot\,;\bar{p}_{2j+1},\wt{p}_{2j+1};\eps)|_{[0,T(\bar{p}_{2j+1},\wt{p}_{2j+1})]}.
\eeq
Let us denote
\[
D_{2j+1}:=\{p_{2j+1} \in \left(\pa B_R(0) \cap \bar{U}\right): |\bar{p}_{2j}-p_{2j+1}| \leq \d\}
\]
and let $D_{2j+1}^{\circ}$ denote its interior. We define $G_{2j+1}:D_{2j+1} \to \R$ as
\begin{multline*}
G_{2j+1}(p_{2j+1}):=L\left([0,T(p_{2j+1})]; y_{\text{ext}}(\cdot\,;\bar{p}_{2j},p_{2j+1};\eps)\right)\\
+ L\left([0,T(p_{2j+1},\wt{p}_{2j+1})]; \wt{y}(\cdot\,;p_{2j+1},\wt{p}_{2j+1};\eps)\right),
\end{multline*}
where we write (and we will adopt this notation from now on) $T(p_{2j+1})$ for \\$T_{\text{ext}}(\bar{p}_{2j},p_{2j+1};\eps)$. 
Of course, with minor changes we can also define a function $G_{2j}$, for every $j \in \{0,\ldots,n-1\}$. 
Note that $G_k$ is continuous (for every $k$), for it is the sum of continuous terms with respect to $p_k$. As a consequence, $G_k$ has a minimum. The following statement can be easily proven.

\begin{lemma}\label{localizzazione minimi F}
If $(\bar{p}_0,\ldots,\bar{p}_{2n})$ is a minimizer for $F$, then $\bar{p}_k$ is a minimizer for $G_k$.
\end{lemma}

Contrary to $F$, $G_k$ is differentiable for every $k$: let's think at $k=2j+1$; $L\left([0,T(p_{2j+1})]; y_{\text{ext}}(\cdot\,;\bar{p}_{2j},p_{2j+1};\eps)\right)$ depends smoothly on $p_{2j+1}$ for the differentiable dependence of outer solutions with respect to the ends, and the length $L\left([0,T(p_{2j+1},\wt{p}_{2j})]; \wt{y}(\cdot\,;p_{2j+1},\wt{p};\eps)\right)$ depends smoothly on $p_{2j+1}$ for the differentiable dependence of minimal geodesics in a strongly convex neighbourhood with respect to the ends. Therefore the minimality of $\bar{p}_{2j+1}$ implies that  
\[
\bar{p}_{2j+1} \in D_{2j+1}^{\circ} \quad \Rightarrow \quad \frac{\pa G_{2j+1}}{\pa p_{2j+1}}(\bar{p}_{2j+1})=0;
\]
Next we show that, if $\eps$ is small enough, the minimizer $\bar{p}_k$ lies in the interior $D_k^{\circ}$ for every $k$. Moreover, 
and that the stationarity condition for $G_k$ provide smoothness of the functions 
\[
\sigma_{2j}(t):= \begin{cases} 
y_{P_{k_{j-1}}}(t;\bar{p}_{2j-1},\bar{p}_{2j};\eps) \qquad \text{if }t \in [0,T(\wt{p}_{2j},\bar{p}_{2j})] \\
y_{\text{ext}}(t-T(\wt{p}_{2j},\bar{p}_{2j});\bar{p}_{2j},\bar{p}_{2j+1};\eps) \\
\text{if } t \in [T(\wt{p}_{2j},\bar{p}_{2j}),T(\wt{p}_{2j},\bar{p}_{2j})+T(\bar{p}_{2j+1})]
\end{cases}
\]
and 
\[
\sigma_{2j+1}(t):= \begin{cases} y_{\text{ext}}(t;\bar{p}_{2j},\bar{p}_{2j+1};\eps) \qquad t \in [0,T(\bar{p}_{2j+1})]\\
y_{P_{k_{j+1}}}(t-T(\bar{p}_{2j+1});\bar{p}_{2j},\bar{p}_{2j+1};\eps) \\
\text{if } t \in [T(\bar{p}_{2j+1}),T(\bar{p}_{2j+1})+T(\bar{p}_{2j+1},\wt{p}_{2j+1})].
\end{cases}
\]
Observing that $\sigma_k$ is (up to a time translation) the restriction of $\gamma_{(\bar{p}_0, \ldots,\bar{p}_{2n})}$ on a neighbourhood of the junction time 
$\mathfrak{T}_{k-1}$, we obtain $\mathcal{C}^1$ regularity for $\gamma_{(\bar{p}_0, \ldots,\bar{p}_{2n})}$ in a neighbourhood of the set of the junction times. 
With this, it won't be difficult to conclude the proof of Theorem 5.3. 
First of all, we can adapt the computations of the partial derivatives developed in Step 3) of the quoted paper with minor changes, obtaining 
\begin{lemma}\label{calcolo delle derivate parziali}
For every $p_{2j} \in D_{2j}$ and for every $\f \in T_{p_{2j}}(B_R(0))$ we have
\[
\frac{\pa G_{2j}}{\pa p_{2j}}(p_{2j})[\f] = \frac{1}{\sqrt{2}}\langle \dot{\wt{y}}(T(\wt{p}_{2j},p_{2j});\wt{p}_{2j},p_{2j};\eps) - \dot{y}_{\text{ext}}(0;p_{2j},\bar{p}_{2j+1};\eps), \f \rangle .
\]
For every $p_{2j+1} \in D_{2j+1}$ and for every $\f \in T_{p_{2j+1}}(B_R(0))$ we have
\[
\frac{\pa G_{2j+1}}{\pa p_{2j+1}}(p_{2j+1})[\f] = \frac{1}{\sqrt{2}}\langle \dot{y}_{\text{ext}}(T(p_{2j+1});\bar{p}_{2j},p_{2j+1};\eps) -\dot{\wt{y}}(0;p_{2j+1},\wt{p}_{2j+1};\eps), \f \rangle .
\]
\end{lemma}

The next Lemma replaces Step 4) of the proof of Theorem 5.3: its role is to prove that the minimizer falls naturally in the interior of the constraint $D_{k}$.

\begin{lemma}\label{minimi interni}
There exists $\bar{\eps}>0$ such that for every $\eps \in (0,\bar{\eps})$
\[
\text{$\bar{p}_k$ minimizes $G_k$} \Rightarrow \bar{p}_k \in D_k^\circ \qquad \forall k.
\]
The value $\bar{\eps}$ does not depend neither on $n$ nor on the sequence of partitions $(\mathcal{P}_{k_1},\ldots,\mathcal{P}_{k_n}) \in \mathcal{P}^n$.
\end{lemma}
\begin{proof}
Assume that there exists $k \in \{0,\ldots, 2n\}$ such that 
\[
\begin{cases}
|\bar{p}_k-\bar{p}_{k+1}|=\d & \text{if $k$ is even}\\
|\bar{p}_k-\bar{p}_{k-1}|=\d & \text{if $k$ is odd}.
\end{cases}
\]
To fix our minds, let $k=1$. We can produce an explicit variation of $\bar{p}_1$ such that $G_1$ decreases along this variation, in contradiction with the minimality of $\bar{p}_1$.
Let's write
\begin{align*}
& y_{\text{ext}}(t;p_0,p_1;\eps)= r_{\text{ext}}(t;p_0,p_1;\eps) \exp\{i \t_{\text{ext}}(t;p_0,p_1;\eps)\}, \\
& y_{P_{k_1}}(t;p_1,p_2;\eps)= r_{P_{k_1}}(t;p_1,p_2;\eps) \exp\{i \t_{P_{k_1}}(t;p_1,p_2;\eps)\}, \\
& \wt{y}(t;p_1,\wt{p}_1;\eps) = \wt{r}(t;p_1,\wt{p}_1;\eps) \exp\{i \wt{\t}(t;p_1,\wt{p}_1;\eps)\}.
\end{align*}
The first step consists in proving that there are $C_1>0$ and $\eps_4>0$ such that if $0<\eps<\eps_4$ then
\begin{multline}\label{mom ang grande}
|\dot{\t}_{\text{ext}}(T_{\text{ext}}(p_*,p_{**};\eps);p_*,p_{**};\eps)| \geq C_1 \qquad \text{for every} \\
(p_*,p_{**}) \in \{(p_*,p_{**}) \in (\pa B_R(0))^2: |p_*-p_{**}|=\d\}.
\end{multline}
This means that, if the distance between $(p_*,p_{**})$ is $\d$, for $\eps$ small enough the outer solution connecting these two points arrive in $p_{**}$ with an angular momentum which cannot be too small.
To show it, we observe that, since the unperturbed problem ($\eps=0$) is invariant under rotations, there is $C_2>0$ such that
\begin{multline*}
|\dot{\t}_{\text{ext}}(T_{\text{ext}}(p_*,p_{**};0);p_*,p_{**};0)|=C_2 \qquad \text{for every} \\
(p_*,p_{**}) \in \{(p_*,p_{**}) \in (\pa B_R(0))^2: |p_*-p_{**}|=\d\}.
\end{multline*}
Now, assume by contradiction that \eqref{mom ang grande} does not hold. Then there exist two sequences $(\l_n)$ and $(\eps_n)$ of positive numbers and a sequence of points $(p_*^n,p_{**}^n) \in \left(\pa B_R(0)\right)^2$, with $|p_*^n-p_{**}^n|=\d$ for every $n$, such that
\[
\l_n \to 0 \quad \eps_n \to 0 \quad |\dot{\t}_{\text{ext}}(T_{\text{ext}}(p_*^n,p_{**}^n;\eps_n);p_*^n,p_{**}^n;\eps_n)|<\l_n.
\]
Since the set $\{  (p_*,p_{**}) \in (\pa B_R(0))^2: |p_*-p_{**}|=\d\}$ is compact, up to a subsequence $(p_*^n,p_{**}^n)$ converges to a point $(\bar{p}_*,\bar{p}_{**})$, and by applying the continuous dependence of the outer solutions with respect to variations of the vector field and initial data we would obtain
\[
|\dot{\t}_{\text{ext}}(T_{\text{ext}}(\bar{p}_*,\bar{p}_{**};0);\bar{p}_*,\bar{p}_{**};0)|=0,
\]
a contradiction. This proves \eqref{mom ang grande}.
On the other hand, we can prove that each inner trajectory (for every $p_1$ and $p_2$ on $\pa B_R(0)$, for every $P_j \in \mathcal{P}$) starts with a small angular momentum, if $\eps$ is sufficiently small; to be precise 
\beq\label{momento angolare piccolo}
\forall \lambda > 0 \ \exists \eps_5 > 0: 0<\eps<\eps_5 \Rightarrow |\dot{\t}_{P_j}\left(0;p_1,p_2;\eps\right)|<\lambda,
\eeq
for every $p_1,p_2 \in \pa B_R(0)$, for every $P_j \in \mathcal{P}$. 
To show it, we define $S=S(p_1,p_2;\eps) \in \R^+$ by
\[
t \in (0,S) \Rightarrow \frac{R}{2}<|y_{P_j}(t;p_1,p_2;\eps)| <R \text{ and } |y_{P_j}(S;p_1,p_2;\eps)|=\frac{R}{2}.
\]
The energy integral makes this quantity uniformly bounded from below by a positive constant $C$, as function of $\eps$. Letting $\eps \to 0^+$ the centres collapse in the origin, so that for the angular momentum of $y_{P_{k_1}}(\cdot\,;p_1,p_2;\eps)$ it results
\[
\mathfrak{C}_{y_{P_{k_1}}(\cdot\,;p_1,p_2;\eps)}\left(t\right) = o(1)  \quad \text{for} \quad \eps \to 0^+,
\]
uniformly in $[0,C]$ (recall Proposition 4.20). This limit is uniform in $p_1$, $p_2$ and $P_{k_1}$: since the curve parametrized by $y_{P_{k_1}}(\cdot\,;p_1,p_2;\eps)$ has to pass inside the ball or radius $\eps$, the function $y_{P_{k_1}}(\cdot\,;p_1,p_2;\eps)$ uniformly converges in $[0,C]$, for $\eps \to 0$, to the same (up to a rotation) piece of collision solution of the Kepler's problem. This proves \eqref{momento angolare piccolo}. The choice $\lambda=C_1/2$ in \eqref{momento angolare piccolo} gives
\[
| \dot{\t}_{P_j}(0;p_1,p_2;\eps) | < \frac{C_1}{2} \qquad \text{if $0<\eps<\eps_5$},
\]
for every $p_1,p_2 \in \pa B_R(0)$, for every $P_j \in \mathcal{P}$. Recalling equation \eqref{oss 1}, we deduce that 
\beq\label{**52}
| \dot{\wt{\t}}(0;\bar{p}_1,\wt{p}_1;\eps)| < \frac{C_1}{2} \qquad \text{if $0<\eps<\eps_5$}.
\eeq
Assume now that $\bar{p}_0=R \exp\{i \bar{\t}_0\}$, $\bar{p}_1=R\exp\{i\bar{\t}_1\}$, with $\bar{\t}_0, \bar{\t}_1 \in [0,2\pi)$ and $\bar{\t}_0 < \bar{\t}_1$ (if $\bar{\t}_0 < \bar{\t}_1$ a very similar argument works). We consider a variation $\f \in T_{\bar{p}_1}(\pa B_R(0))$ of $\bar{p}_1$ directed towards $\bar{p}_0$ on $\pa B_R(0)$. Since $\bar{\t}_0<\bar{\t}_1$, this variation is a positive multiple of $-i\exp{\{i \bar{\t}_1\}}$. Collecting \eqref{mom ang grande}, \eqref{**52} and using Lemma \ref{calcolo delle derivate parziali}, for any $0<\eps<\min\{\eps_2,\eps_3,\eps_4,\eps_5\}=: \bar{\eps}$ we have that if $|\bar{p}_0-\bar{p}_1|=\delta$ then
\begin{multline*}
\frac{\pa G_1}{\pa p_1} (\bar{p}_1)[\f]\\
=\frac{CR}{\sqrt{2}}\left\langle  \left(\dot{\t}_{\text{ext}}\left(T_{\text{ext}}(\bar{p}_0,\bar{p}_1;\eps);\bar{p}_0,\bar{p}_1;\eps\right) - \dot{\wt{\t}}\left(0;\bar{p}_1,\wt{p}_1;\eps\right)\right) i e^{i \t_1}, -i e^{i \t_1}\right\rangle\\
 <\frac{CR}{\sqrt{2}}\left(\frac{C_1}{2}- C_1\right)<0,
\end{multline*}
against the minimality of $(\bar{p}_0,\ldots,\bar{p}_{2n})$. We point out that $\bar{\eps}$ does not depend neither on $n \in \mathbb{N}$ nor on $(P_{k_1},\ldots,P_{k_n}) \in \mathcal{P}^n$.
\end{proof}

As a consequence, we get the counterpart of Step 5) of the proof of Theorem 5.3:
\begin{lemma}\label{regolarita' dei minimi}
If $0<\eps< \bar{\eps}$, then each function $\sigma_{k}$ is $\mathcal{C}^1$.
\end{lemma}
The conclusion of the proof of Theorem 5.3, Step 6), remains the same.

\begin{rem}
We saw that the extremality condition for $(\bar{p}_0,\ldots,\bar{p}_{2n})$ implies that
\[
\dot{y}_{2j}(0)=\dot{y}_{2j+1}(T_{2j+1}) \quad \text{and}\quad \dot{y}_{2j+1}(0)=\dot{y}_{2j}(T_{2j}) \quad \forall j=0,\ldots,n-1.
\]
Therefore, for the uniqueness of the outer arcs and of the solutions of regular Cauchy problem, $y_{2j+1}$ is uniquely determined in $[0,t_0]$, where $t_0$ is the first collision time of $y_{2j+1}$; also, $y_{2j+1}$ is uniquely determined in $[t_1,T_{2j+1}]$, where $t_1$ is the last collision time of $y_{2j+1}$. Since every inner minimizer has at most one collision, if $y_{2j+1}$ connects $\bar{p}_{2j+1}$ and $\bar{p}_{2j+2}$, where $(\bar{p}_0,\ldots,\bar{p}_{2n})$ minimizes $F$, then it is uniquely determined. In particular, $F$ turns out to be differentiable with respect to the ends.\end{rem}

\end{document}